\documentclass[a4paper, 12pt]{article}
\usepackage{}

\usepackage{mathrsfs}
\usepackage{epsfig}
\usepackage{amsmath}
\usepackage{amssymb}
\usepackage{latexsym}
\usepackage{amsfonts}
\usepackage{amsthm}

\usepackage[numbers,sort&compress]{natbib}
\usepackage{graphicx}
\ifx\pdfoutput\undefined  \else
 \fi

\usepackage[centerlast]{caption2}

\usepackage{color}

\usepackage{txfonts}
\usepackage{amssymb}
\usepackage{mathrsfs}
\usepackage{amsfonts}

\newtheorem{theorem}{Theorem}[section]
\newtheorem{lemma}[theorem]{Lemma}

\newtheorem{remark}[theorem]{Remark}
\newtheorem{definition}[theorem]{Definition}

\newtheorem{problem}[theorem]{Problem}

\usepackage{latexsym,amssymb}
\usepackage{amsmath}

\newcommand{\F}{{\mathbb F}}

\begin{document}

\title{{Structure of long idempotent-sum free sequences over finite cyclic semigroups}}
\author{
Guoqing Wang\\
\small{Department of Mathematics, Tianjin Polytechnic University, Tianjin, 300387, P. R. China}\\
\small{Email: gqwang1979@aliyun.com}
\\
}
\date{}
\maketitle

\begin{abstract}
Let $\mathcal{S}$ be a finite cyclic semigroup written additively. An element $e$ of $\mathcal{S}$ is said to be idempotent if $e+e=e$. A sequence $T$ over $\mathcal{S}$ is called {\sl idempotent-sum free} provided that no idempotent of $\mathcal{S}$ can be represented as a sum of one or more terms from $T$. We prove that an idempotent-sum free sequence over $\mathcal{S}$ of length over approximately a half of the size of $\mathcal{S}$ is well-structured. This result generalizes the Savchev-Chen Structure Theorem for zero-sum free sequences over finite cyclic groups.
\end{abstract}


\section {Introduction}

Let $G$ be a finite cyclic group of order $n$, and let $T$ be a sequence of terms from $G$.
We say that $T$ is zero-sum free, if no nonempty subsequence of $T$ sums
to zero (the identity element of $G$). An easy observation shows that $n-1$ is the maximal length of
a zero-sum free sequence and if $T$ is a zero-sum free sequence of length exactly
$n-1$, then $T$ consists of one element $g$ which is repeated $n-1$ times. Investigations of the structure of long zero-sum free
sequences were started in 1970s' by J.D. Bovey, P. Erd\H{o}s and I. Niven \cite{BoErNi}, who proved that any long zero-sum free sequence of length $\ell$ over the cyclic group of order $n$ must contain one term with repetitions at least $2\ell-n+1$. After that, the study of the structure of long zero-sum free sequences over cyclic groups
has attracted considerable attention (see \cite{GaoInteger, GaoGeroldingerCombintaorica, GaoLiYuanZhuang, GeroHamidoune, SavchevChen, Yuan}), among which for a cyclic group of order $n$, W. Gao \cite{GaoInteger} characterized the zero-sum free sequence of length roughly greater than $\frac{2n}{3}$,  S. Savchev and F. Chen \cite{SavchevChen}  proved that each zero-sum free sequence with length greater than $\frac{n}{2}$ must have a regular structure. The Savchev-Chen structural zero-sum theorem (see \cite{Grynkiewiczmono}, Chapter 11) found a variety of applications in problems of zero-sum theory.

Problems of this flavor are known as inverse zero-sum problems, which ask for an structural description of sequences without the desired prescribed properties. Often the study of inverse problems is complex. For example, even for the simplest noncyclic case--rank 2 groups (the direct sum of two finite cyclic groups), it is only recently, and with a massive amount combined efforts of W. Gao, A. Geroldinger, D. Grynkiewicz and W.A. Schmid in several multi-pages articles (see \cite{GaoGeroldingerPeriodic,GaoGeroldingerInteger,  GaoGeroldingerDavidIII, GaoGeroldingerSchmidInverse, SchmidII}), and of C. Reiher \cite{Reiher} in a final step, that the structure of zero-sum free sequences of the maximal length has been determined.
Besides of being of interest for its own right, the inverse zero-sum problems are particularly relevant to the theory of non-unique factorizations. For this connection and some general information
on factorization theory we refer to \cite{CziszterDoGerolding,GH,GRuzsa}.

Another motivation of this manuscript comes from the question (see [4, 15]) proposed by P. Erd\H{o}s to D.A. Burgess which can restated as follows:

``{\sl Let $\mathcal{S}$ be a finite semigroup of order $n$. Is there an idempotent-product free sequence of terms from $\mathcal{S}$ with length $n$? }"

A sequence $T$ over $\mathcal{S}$ is called {\sl idempotent-product free} provided that no idempotent of $\mathcal{S}$ can be represented as a product of one or more terms from $T$ in any order.  The Erd\H{o}s' question was answered partially by Burgess \cite{Burgess69} in the case when $\mathcal{S}$ is commutative or contains only one idempotent, and was completely affirmed by
D.W.H. Gillam, T.E. Hall and N.H. Williams \cite{Gillam72},  and was extended to infinite semigroups by the author \cite{wangidempotent}.
Note that if the semigroup $\mathcal{S}$ is commutative, we use {\sl idempotent-sum free} for {\sl idempotent-product free} since the operation is addition and the order of terms in additions does not matter.  In particular, if $\mathcal{S}$ is a finite abelian group,
the notion {\sl idempotent-sum free} reduces to be {\sl zero-sum free}, because the identity element is the unique idempotent in a group.

Recently  some zero-sum type problems were investigated in the setting of commutative semigroups (see \cite{Deng, wangDavenportII, wangAddtiveirreducible, wang-zhang-qu, wang-zhang-wang-qu} for example).
In this manuscript we show that an idempotent-sum free sequence over a finite cyclic semigroup of length over approximately a half of the size of the cyclic semigroup will yield a regular structure, which generalizes the Savchev-Chen Structure Theorem for zero-sum free sequences over finite cyclic groups.  Moreover, we investigate the Ramsey-type question to determine the least integer $\ell$ to ensure that an idempotent-sum free sequence of length at least $\ell$ will have this regular structure, meanwhile, we also extend an invariant proposed by S.T. Chapman, M. Freeze and W.W. Smith \cite{ChFrSmi,ChSmi} on minimal zero-sum sequences into finite cyclic semigroups.

\section{Notation and terminologies}

For integers $a,b\in \mathbb{Z}$, we set $[a,b]=\{x\in \mathbb{Z}: a\leq x\leq b\}$.
For a real number $x$, we denote by $\lfloor x\rfloor$ the largest integer that is less
than or equal to $x$, and by $\lceil x\rceil$ the smallest integer that is greater than or equal to $x$.

Let $\mathcal{S}$ be a commutative semigroup written additively, where the operation is denoted as $+$.
For any positive integer $m$ and any element $a\in \mathcal{S}$, we denote by $ma$ the sum $\underbrace{a+\cdots+a}\limits_{m}$. An element $e$ of $\mathcal{S}$ is said to be idempotent if $e+ e=e$. A cyclic semigroup is a semigroup generated by a single element $s$, denoted $\langle s\rangle$, consisting all elements which can be represented as $m s$ for some positive integer $m$.
If the cyclic semigroup $\langle s\rangle$ is infinite then $\langle s\rangle$ is isomorphic to the semigroup of $\mathbb{N}$ with addition (see \cite{Grillet monograph}, Proposition 5.8), and if $\langle s\rangle$ is finite
then the least integer $k>0$ such that $ks=ts$ for some positive integer $t\neq k$ is called the {\bf index} of $\langle s\rangle$,  then the least integer $n>0$ such that $(k+n)s=k s$ is called the {\bf period} of $\langle s\rangle$. We denote a finite cyclic semigroup of index $k$ and period $n$ by $C_{k; n}$. In particular, if $k=1$ the semigroup $C_{k; n}$ reduces to a cyclic group of order $n$.
For any element $a$ of $C_{k; n}=\langle s\rangle$, let ${\rm ind}_{s}(a)$ (write as ${\rm ind}(a)$ for simplicity) be the least positive integer $t$ such that $t s=a$. Let $\mathbb{Z}\diagup n \mathbb{Z}$ be the additive group of integers modulo $n$.
Define a map $$\psi:C_{k; n}\rightarrow \mathbb{Z}\diagup n \mathbb{Z} \ \ \ \mbox{ given by } \ \
\psi: a\mapsto {\rm ind}(a)+ n\mathbb{Z} \ \ \ \ \mbox{ for all } a\in C_{k; n}.$$

We also need to introduce notation and terminologies on sequences over semigroups and follow the notation of A. Geroldinger, D.J. Grynkiewicz and
others used for sequences over groups (cf. [\cite{Grynkiewiczmono}, Chapter 10] or [\cite{GH}, Chapter 5]). Let ${\cal F}(\mathcal{S})$
be the free commutative monoid, multiplicatively written, with basis
$\mathcal{S}$. We denote multiplication in $\mathcal{F}(\mathcal{S})$ by the boldsymbol $\cdot$ and we use brackets for all exponentiation in $\mathcal{F}(\mathcal{S})$. By $T\in {\cal F}(\mathcal{S})$, we mean $T$ is a sequence of terms from $\mathcal{S}$ which is
unordered, repetition of terms allowed. Say
$T=a_1a_2\cdot\ldots\cdot a_{\ell}$ where $a_i\in \mathcal{S}$ for $i\in [1,\ell]$.
The sequence $T$ can be also denoted as $T=\mathop{\bullet}\limits_{a\in \mathcal{S}}a^{[{\rm v}_a(T)]},$ where ${\rm v}_a(T)$ is a nonnegative integer and
means that the element $a$ occurs ${\rm v}_a(T)$ times in the sequence $T$. By $|T|$ we denote the length of the sequence, i.e., $|T|=\sum\limits_{a\in \mathcal{S}}{\rm v}_a(T)=\ell.$ By $\varepsilon$ we denote the
{\sl empty sequence} over $\mathcal{S}$ with $|\varepsilon|=0$. We call $T'$
a subsequence of $T$ if ${\rm v}_a(T')\leq {\rm v}_a(T)\ \ \mbox{for each element}\ \ a\in \mathcal{S},$ denoted by $T'\mid T,$ moreover, we write $T^{''}=T\cdot  T'^{[-1]}$ to mean the unique subsequence of $T$ with $T'\cdot T^{''}=T$.  We call $T'$ a {\sl proper} subsequence of $T$ provided that $T'\mid T$ and $T'\neq T$. In particular, the  empty sequence  $\varepsilon$ is a proper subsequence of every nonempty sequence. We say $T_1,\ldots,T_{m}$ are {\sl disjoint subsequences} of $T$ provided that $T_1\cdot \ldots\cdot T_{m}\mid T$.
Let $\sigma(T)=a_1+\cdots +a_{\ell}$ be the sum of all terms from $T$.
Let $\Sigma(T)$ be the set consisting of the elements of $\mathcal{S}$ that can be represented as a sum of one or more terms from $T$, i.e., $\Sigma(T)=\{\sigma(T'): T' \mbox{ is taken over all nonempty subsequences of }T\}$.
We call $T$ a {\bf zero-sum} sequence provided that $\mathcal{S}$ has an identity $0_{\mathcal{S}}$ and $\sigma(T)=0_{\mathcal{S}}$.
In particular,
if $\mathcal{S}$ has an identity $0_{\mathcal{S}}$,  we adopt the convention
that $\sigma(\varepsilon)=0_\mathcal{S}.$
We say  the sequence $T$ is
\begin{itemize}
     \item a {\bf zero-sum free sequence} if $T$ contains no nonempty zero-sum subsequence;
     \item a {\bf minimal zero-sum sequence} if $T$ is a nonempty zero-sum sequence and and $T$ contains no nonempty proper zero-sum subsequence;
     \item an {\bf idempotent-sum sequence} if $\sigma(T)$ is an idempotent;
     \item an {\bf idempotent-sum free sequence} if $T$ contains no nonempty idempotent-sum subsequence;
     \item a {\bf minimal idempotent-sum sequence} if $T$ is a nonempty idempotent-sum sequence and $T$ contains no nonempty proper idempotent-sum subsequence.
\end{itemize}
It is worth remarking that when the commutative semigroup $\mathcal{S}$ is an abelian group, the notion {\sl zero-sum free sequence} and {\sl idempotent-sum free sequence} make no difference, and the similar holds for the pair of terminologies {\sl minimal zero-sum sequence} and {\sl minimal idempotent-sum sequence}.

For a finite cyclic semigroup $C_{k; n}$, we extend $\psi$ to the map $$\Psi:\mathcal{F}(C_{k; n})\rightarrow \mathcal{F}(\mathbb{Z}\diagup n \mathbb{Z}) \ \ \ \mbox{ given by } \ \  \Psi: T\mapsto \mathop{\bullet}\limits_{a\mid T} \psi(a)\ \ \ \ \mbox{ for any sequence }T\in \mathcal{F}(C_{k; n}).$$

\section{Structure of long idempotent-sum free sequences}

To give the main theorem, we shall need some preliminaries.

\begin{lemma} (Folklore) \label{lemma n-1 or n zero-sum free}\ Let $G$ be a cyclic group of order $n\geq 2$. If $T\in \mathcal{F}(G)$ is either a zero-sum free sequence of length $n-1$ or a minimal zero-sum sequence of length $n$, then $T=g^{[n-1]}$ or resp. $T=g^{[n]}$ for some $g\in G$ with ${\rm ord}(g)=n$.
\end{lemma}

\begin{lemma}\label{Lemma decomposition} \ Let $G$ be a cyclic group of order $n\geq 2$, and let $T\in \mathcal{F}(G\setminus \{0_G\})$ be a sequence of length at least $n-1$. Let $U$ be one of the longest zero-sum subsequences of $T$. If $|T\cdot U^{[-1]}|=n-1$ then $T=g^{[|T|]}$ for some $g\in G$ with ${\rm ord}(g)=n$.
 \end{lemma}

\begin{proof}  We take disjoint minimal zero-sum subsequences of $T$, say $U_1,\ldots, U_{\ell}$, with $\ell$ being {\bf maximal}. Let $U_0=T\cdot (U_1\cdot \ldots\cdot U_{\ell})^{[-1]}$. By the maximality of $\ell$, we have $U_0$ is zero-sum free. It follows that $|U_0|\leq {\rm D}(G)-1=n-1$, which implies that $U=U_1\cdot \ldots\cdot U_{\ell}$ is one of the longest zero-sum subsequence of $T$ and $|U_0|=n-1$. It follows from Lemma \ref{lemma n-1 or n zero-sum free} that $$U_0=z^{[n-1]}$$  for some $z\in G$ with ${\rm ord}(z)=n$.  Now it suffices to assume that $|T|\geq n$, i.e., $\ell>0$, and prove that
$\alpha=z$ for any term $\alpha \mid U$.

Take arbitrary $\theta\in [1,\ell]$ and take an arbitrary term $x$ of $U_{\theta}$. Since $|x\cdot U_0|=n={\rm D}(G)$ and $U_0$ is zero-sum free, it follows that
$x\cdot U_0$ contains a minimal zero-sum subsequence  $U_{\theta}^{'}$  with $x\mid U_{\theta}^{'}.$
Let $U_0^{'}=(U_{\theta}\cdot U_0) \cdot {U_{\theta}^{'}}^{[-1]}$. Observe that $U_1,\ldots, U_{\theta-1}, U_{\theta}^{'},U_{\theta+1},\ldots,U_{\ell}$ are disjoint minimal zero-sum subsequences of $T$, and that $U_0^{'}=T\cdot (U_1\cdot \ldots\cdot  U_{\theta-1}\cdot U_{\theta}^{'}\cdot U_{\theta+1}\cdot \ldots\cdot U_{\ell})^{[-1]}$.
By the maximality of $\ell$, we see that $U_0^{'}$ is also a zero-sum free sequence and
\begin{equation}\label{equation length of U0'}
|U_0^{'}|=n-1.
\end{equation}

Suppose $|U_{\theta}^{'}|=n$. Then $U_{\theta}^{'}=x\cdot U_0$. It follows from Lemma \ref{lemma n-1 or n zero-sum free} that $x=z.$

Suppose $|U_{\theta}^{'}|<n.$ Observe that $U_0^{'}=((x\cdot U_0)\cdot {U_{\theta}^{'}}^{[-1]})\cdot (U_{\theta}\cdot x^{[-1]})$  and
$(x\cdot U_0) \cdot {U_{\theta}^{'}}^{[-1]}$ is a nonempty subsequence of $U_0$. Since all terms of $T$ are nonzero, it follows that
\begin{equation}\label{equation |Utheta|>1}
|U_{\theta}|>1
\end{equation}
and $|U_{\theta}\cdot x^{[-1]}|>0$.
By \eqref{equation length of U0'} and Lemma \ref{lemma n-1 or n zero-sum free}, we derive that $y=z$ for every term  $y\mid U_{\theta}\cdot x^{[-1]}.$
By \eqref{equation |Utheta|>1}, we have that as $x$ takes every term of $U_{\theta}$, so does $y$.  By the arbitrariness of choosing $\theta$ from $[1,\ell]$ and the arbitrariness of choosing $x$ from $U_{\theta}$, we have the lemma proved.
\end{proof}

\begin{lemma}\label{Lemma cyclic semigroup} (\cite{Grillet monograph},  Chapter I) \ Let $\mathcal{S}=C_{k; n}$ be a finite cyclic semigroup generated by the element $s$. Then  $\mathcal{S}=\{s,\ldots,k s,(k+1)s,\ldots,(k+n-1)s\}$
with
$$\begin{array}{llll} & is+js=\left \{\begin{array}{llll}
               (i+j)s, & \mbox{ if } \  i+j \leq  k+n-1;\\
                ts, &  \mbox{ if }  \ i+j \geq k+n, \ \mbox{ where}  \  k\leq t\leq k+n-1 \ \mbox{ and } \ t\equiv i+j\pmod{n}. \\
              \end{array}
           \right. \\
\end{array}$$
Moreover, there exists a unique idempotent, say $\ell s$, in the cyclic semigroup $\langle s\rangle$, where $$\ell\in [k,k+n-1] \  \mbox{ and }\  \ell\equiv 0\pmod {n}.$$
\end{lemma}

By Lemma \ref{Lemma cyclic semigroup}, it is easy to derive the following.

\begin{lemma}\label{Lemma product condition containing idmepotent} \  Let $\mathcal{S}=C_{k; n}$, and let $W\in \mathcal{F}(\mathcal{S})$ be a nonempty sequence. Then $W$ is an idempotent-sum sequence if, and only if,
$\sum\limits_{a\mid W}{\rm ind}(a)\geq \left\lceil\frac{k}{n}\right\rceil n$ and $\sum\limits_{a\mid W}{\rm ind}(a)\equiv 0\pmod{n}$.
\end{lemma}

\begin{definition}\label{Definition behaving sequence} (\cite{GRuzsa}, Definition 5.1.3) \ Let $G$ be an abelian group. Let $T\in \mathcal{F}(G)$ with $T=(n_1 g)\cdot \ldots\cdot (n_{\ell} g)$, where $\ell=|T|\in \mathbb{N}$, $g\in G$, $1=n_1\leq \cdots \leq n_{\ell}$, $n=n_1+\cdots +n_{\ell}\leq {\rm ord}(g)$ and $\sum(T)=\{g,2g,\ldots,ng\}$.
If $n<{\rm ord}(g)$ we call $T$ smooth (zero-sum free smooth in full), and if $n={\rm ord}(g)$ we call $T$ zero-sum smooth.  In the case we say more precisely that $T$ is (zero-sum free) $g$-smooth and zero-sum $g$-smooth respectively.
\end{definition}

We remark that in Definition 5.1.3 of \cite{GRuzsa}, the notion `smooth' is used only for zero-sum free sequences. To describe both idempotent-sum free sequences and minimal idempotent-sum sequences in this paper, we extend the notion as above and define zero-sum smooth sequences. Note that for any nonempty sequence $T\in \mathcal{F}(\mathbb{Z})$, $T$ is $1$-smooth if and only if $\sum(T)=[1,\sum\limits_{a\mid T} a]$.

\begin{lemma} \label{Lemma behaving for integers} \ For $\ell\geq 1$, let $T=\mathop{\bullet}\limits_{i\in [1,\ell]} h_i\in \mathcal{F}(\mathbb{Z})$ where $h_i>0$ for each $i\in [1,\ell]$. Suppose that the sequence $T$ is not $1$-smooth. Then $\sum\limits_{i=1}^{\ell} h_i\geq 2\ell$. Moreover, the equality $\sum\limits_{i=1}^{\ell} h_i= 2\ell$ holds if and only if either $T=1^{[\ell-1]}\cdot (\ell+1)$ or $T=2^{[\ell]}.$
\end{lemma}

 \begin{proof}  By induction on $\ell$. If $\ell=1$, the conclusion is obvious. Hence, we assume that the conclusion holds true for all $\ell<m$ with $m\geq 2$. Consider the case of $\ell=m$.

 Suppose $h_i\geq 2$ for all $i\in [1,m]$. Then $\sum\limits_{i=1}^{m} h_i\geq 2m$ and equality holds if and only if $T=2^{[m]}.$ Hence, we may assume without loss of generality that $h_1=1$.

 Let $\lambda$ be the largest length of $1$-smooth subsequences of $T$. Notice that
 \begin{equation}\label{equation lambda in}
 1\leq \lambda\leq m-1.
 \end{equation}
 Say $\mathop{\bullet}\limits_{i\in [1,\lambda]} h_i$ is $1$-smooth. Since $\sum(\mathop{\bullet}\limits_{i\in [1,\lambda]} h_i)=[1,\sum\limits_{i\in [1,\lambda]} h_i]$, it follows from the maximality of $\lambda$ that $h_t\geq 2+\sum\limits_{i\in [1,\lambda]} h_i\geq 2+\lambda$ for all $t\in [\lambda+1, m]$.
 Combined with \eqref{equation lambda in}, we have $\sum\limits_{i=1}^{m} h_i=\sum\limits_{i=1}^{\lambda} h_i+\sum\limits_{t=\lambda+1}^{m} h_t\geq \lambda+(m-\lambda)*(2+\lambda)\geq 2m$, moveover, equality $\sum\limits_{i=1}^{m} h_i=2m$ holds if and only if $\lambda=m-1$ and $\sum\limits_{i=1}^{\lambda} h_i=\lambda=m-1$ and $h_m=2+\lambda=m+1$, equivalently, $T=1^{[m-1]}\cdot (m+1)$.
 \end{proof}

\begin{lemma}\label{prop in cyclic semigroup} \ Let $\mathcal{S}={\rm C}_{k;  n}$ with $k>n\geq 1$. Let $T\in \mathcal{F}(\mathcal{S})$ be a sequence of length at least $\frac{(\left\lceil\frac{k}{n}\right\rceil+1) n}{2}-1$.
Then $T$ is idempotent-sum free if, and only if, one of the following conditions holds:

\noindent (i) $\mathop{\bullet}\limits_{a\mid T} {\rm ind}(a)$ is a $1$-smooth sequence with $\sum\limits_{a\mid T} {\rm ind}(a)\leq \left\lceil\frac{k}{n}\right\rceil n-1$;

\noindent (ii) $\mathop{\bullet}\limits_{a\mid T} {\rm ind}(a)=2^{[\frac{(\left\lceil\frac{k}{n}\right\rceil+1) n}{2}-1]}$ with $n\geq 3$ and $\left\lceil\frac{k}{n}\right\rceil n\equiv 1\pmod 2$;

\noindent (iii)  $\mathop{\bullet}\limits_{a\mid T} {\rm ind}(a)=z\cdot 2^{[\lceil\frac{k}{2}\rceil-1]}$ with $n=2$, $z\geq 3$ and $z\equiv 1\pmod 2$;

\noindent (iv)  $\mathop{\bullet}\limits_{a\mid T} {\rm ind}(a)=1^{[\frac{k-3}{2}]}\cdot \frac{k+1}{2}$ with $n=1$ and $k\equiv 1\pmod 2$;

\noindent (v) $\mathop{\bullet}\limits_{a\mid T} {\rm ind}(a)=2^{[\frac{k-1}{2}]}$ with $n=1$ and $k\equiv 1\pmod 2$.
\end{lemma}

\begin{proof}  By Lemma \ref{Lemma product condition containing idmepotent}, it is easy to verify that if the sequence $T$ is given as any one of (i)-(v), then $T$ is idempotent-sum free. Hence, we need only to prove the necessity.

Suppose $\mathop{\bullet}\limits_{a\mid T} {\rm ind}(a)$ is $1$-smooth.
Then $\sum(\mathop{\bullet}\limits_{a\mid T} {\rm ind}(a))=[1,\sum\limits_{a\mid T} {\rm ind}(a)]$.
Since $T$ is idempotent-sum free, it follows that from Lemma \ref{Lemma product condition containing idmepotent} that $\left\lceil\frac{k}{n}\right\rceil n\notin \sum(\mathop{\bullet}\limits_{a\mid T} {\rm ind}(a))$, and so $\sum\limits_{a\mid T} {\rm ind}(a)\leq \left\lceil\frac{k}{n}\right\rceil n-1$. Then (i) holds. Hence, we may assume that $\mathop{\bullet}\limits_{a\mid T} {\rm ind}(a)$ is not $1$-smooth.

Suppose $n=1$. By Lemma \ref{Lemma product condition containing idmepotent} and Lemma \ref{Lemma behaving for integers}, we have that $k-1\geq \sum \limits_{a\mid T} {\rm ind}(a)\geq 2|\mathop{\bullet}\limits_{a\mid T} {\rm ind}(a)|=2|T|\geq k-1$, which implies that $\sum \limits_{a\mid T} {\rm ind}(a)=2|\mathop{\bullet}\limits_{a\mid T} {\rm ind}(a)|=k-1$ and that either (iv) or (v) holds. Hence, we may assume that
\begin{equation}\label{equation n geq 2}
n\geq 2.
\end{equation}

It follows that \begin{equation}\label{equation sum of ind in T geq}
\sum\limits_{a\mid T}  {\rm ind}(a)
\geq  2|\mathop{\bullet}\limits_{a\mid T} {\rm ind}(a)|
\geq 2(\frac{(\left\lceil\frac{k}{n}\right\rceil+1) n}{2}-1)
\geq \left\lceil\frac{k}{n}\right\rceil n.\\
\end{equation}

\noindent {\bf Claim A.} \ {\sl For any nonempty subsequence $L$ of $T$ such that $\sum\limits_{a\mid L} {\rm ind}(a)\geq n-1$, the sequence  $\mathop{\bullet}\limits_{a\mid L} {\rm ind}(a)$ is not  $1$-smooth.}

\noindent {\sl Proof of Claim A.} \ Suppose to the contrary that there exists a nonempty subsequence $L$ of $T$ such that $\sum\limits_{a\mid L} {\rm ind}(a)\geq n-1$ and $\mathop{\bullet}\limits_{a\mid L} {\rm ind}(a)$ is $1$-smooth.  By \eqref{equation sum of ind in T geq}, there exists a subsequence $V$ (perhaps is an empty sequence) of $L$ such that $\sum\limits_{a\mid T\cdot V^{[-1]}} {\rm ind}(a)\equiv 0\pmod n$ and $\sum\limits_{a\mid T\cdot V^{[-1]}} {\rm ind}(a)\geq \left\lceil\frac{k}{n}\right\rceil n$, and thus  by Lemma \ref{Lemma product condition containing idmepotent}, $T\cdot V^{[-1]}$ is a nonempty idempotent-sum subsequence of $T$, which contradicts with $T$ being idempotent-sum free. This proves Claim A.   \qed

By \eqref{equation n geq 2}, we have that $|\Psi(T)|=|T|\geq \frac{(\left\lceil\frac{k}{n}\right\rceil+1) n}{2}-1\geq \frac{3n}{2}-1\geq n={\rm D}(\mathbb{Z}\diagup n \mathbb{Z})$. Let $U$ be
a nonempty subsequence of $T$ such that $\Psi(U)$ is a zero-sum sequence over $\mathbb{Z}\diagup n \mathbb{Z}$, i.e., $\sum\limits_{a\mid U} {\rm ind}(a)$ is a positive multiple of $n$,
with $|U|$ being {\bf maximal}.  Let $W=T\cdot U^{[-1]}$. Then $\Psi(W)$ is either empty or zero-sum free and so
$|W|=|\Psi(W)|\leq {\rm D}(\mathbb{Z}\diagup n \mathbb{Z})-1=n-1$.
Combined with Claim A, Lemma \ref{Lemma product condition containing idmepotent} and Lemma \ref{Lemma behaving for integers}, we conclude that
$$\begin{array}{llll}
(\left\lceil\frac{k}{n}\right\rceil-1) n&\geq &
\sum\limits_{a\mid U} {\rm ind}(a)\\
&\geq &  2|U|=2(|T|-|W|)\\
&\geq & 2(|T|-(n-1)) \\
&\geq & 2(\frac{(\lceil\frac{k}{n}\rceil+1) n}{2}-1-(n-1))\\
&=&(\lceil\frac{k}{n}\rceil-1)n.
\end{array}$$
It follows that
\begin{equation}\label{euqation |T0|=n-1}
|W|=n-1,
\end{equation}
\begin{equation}\label{euqation|T|=}
|T|=\frac{(\lceil\frac{k}{n}\rceil+1) n}{2}-1,
\end{equation}
and
\begin{equation}\label{equation sum U=2|U|}
(\left\lceil\frac{k}{n}\right\rceil-1) n=\sum\limits_{a\mid U} {\rm ind}(a)=2|U|.
\end{equation}
Since $\sum\limits_{a\mid U} {\rm ind}(a)$ is a positive multiple of $n$, it follows from Claim A that $\mathop{\bullet}\limits_{a\mid U} {\rm ind}(a)$ is not $1$-smooth. By Lemma \ref{Lemma behaving for integers} and
\eqref{equation sum U=2|U|}, we have that
$$\mathop{\bullet}\limits_{a\mid U} {\rm ind}(a)=1^{[|U|-1]}\cdot (|U|+1) \mbox{ or } \mathop{\bullet}\limits_{a\mid U} {\rm ind}(a)=2^{[|U|]}.$$

Suppose $n=2$. By Claim A, we see that ${\rm ind}(a)>1$ for each term $a\mid T$. Then $\mathop{\bullet}\limits_{a\mid U} {\rm ind}(a)=2^{[|U|]}$. Combined with \eqref{euqation |T0|=n-1}, we have that ${\rm ind}(z)\geq 3$ and ${\rm ind}(z) \equiv 1\pmod 2$, where $z$ denotes the unique term of $W$. Combined with \eqref{euqation|T|=}, then Condition (iii) holds.

Hence, it remains to consider the case of $$n\geq 3.$$

Suppose $\mathop{\bullet}\limits_{a\mid U} {\rm ind}(a)=1^{[|U|-1]}\cdot (|U|+1)$. By Claim A, we see that $|U|-1<n-1$ and $|U|<n$. Since $2|U|=|U|-1+|U|+1=\sum\limits_{a\mid U} {\rm ind}(a)$ is a positive multiple of $n$, it follows that $|U|=\frac{n}{2}$. Then all terms of $\Psi(T)$ are nonzero because neither of the two terms 1 and $|U|+1=\frac{n}{2}+1$ of the sequence $\mathop{\bullet}\limits_{a\mid U} {\rm ind}(a)$ is congruent to $0$ modulo $n$. By \eqref{euqation |T0|=n-1} and by applying Lemma \ref{Lemma decomposition} with $\Psi(T)$, we conclude that all terms of $\Psi(T)$ are equal, which is a contradiction with $1\not\equiv \frac{n}{2}+1=|U|+1 \pmod n$. Hence,
\begin{equation}\label{equation sum of U}
\mathop{\bullet}\limits_{a\mid U} {\rm ind}(a)=2^{[|U|]}.
\end{equation}

It follows that all terms of $\Psi(T)$ are nonzero. By \eqref{euqation |T0|=n-1} and by applying Lemma \ref{Lemma decomposition} with $\Psi(T)$, we conclude that all terms of the sequence $\Psi(T)$ are equal to a generator of the group $\mathbb{Z}\diagup n \mathbb{Z}$, i.e.,   \begin{equation}\label{equation alpha=m in T}
{\rm ind}(a)\equiv 2\pmod n \ \ \mbox{ for each } a\mid W
\end{equation}
and \begin{equation}\label{equation n equiv 1 mod 2}
n\equiv 1\pmod 2.
\end{equation}
Take an arbitrary term $\alpha$ of $W$ and an arbitrary term $\beta$ of $U$, and set $U'=(U\cdot \beta^{[-1]})\cdot \alpha$. It follows from \eqref{equation sum of U} and \eqref{equation alpha=m in T} that $\Psi(U')$ is also a zero-sum sequence with $|U'|=|U|$. By replacing $U$ with $U'$ and by \eqref{equation sum of U}, we conclude that $\alpha=\beta=2$. By the arbitrariness of choosing $\alpha$, we have that $\mathop{\bullet}\limits_{a\mid T} {\rm ind}(a)=2^{[|T|]}.$ By \eqref{equation sum U=2|U|} and \eqref{equation n equiv 1 mod 2}, we have that $\left\lceil\frac{k}{n}\right\rceil n=(\left\lceil\frac{k}{n}\right\rceil-1) n+n=2|U|+n\equiv 1\pmod 2$. Combined with \eqref{euqation|T|=}, then Condition (ii) holds. This completes the proof of the theorem.
\end{proof}

\begin{lemma} \label{lemma case of k leq n} \ Let $\mathcal{S}={\rm C}_{k;  n}$ with $1\leq k\leq n$, and let $T\in \mathcal{F}(\mathcal{S})$ be a nonempty sequence. Then $T$ is an idempotent-sum free [resp. minimal idempotent-sum] sequence if and only if $\Psi(T)\in \mathcal{F}(\mathbb{Z}\diagup n \mathbb{Z})$ is a zero-sum free [resp. minimal zero-sum] sequence.
 \end{lemma}

 \begin{proof} \ Since $k\leq n$, we see that $\sum\limits_{a\mid W}  {\rm ind}(a)\equiv 0\pmod n$ implies $\sum\limits_{a\mid W}  {\rm ind}(a)\geq n=\lceil \frac{k}{n}\rceil n$ for any nonempty sequence $W\in \mathcal{F}(\mathcal{S})$. Then the lemma follows from Lemma \ref{Lemma product condition containing idmepotent} and the definition of the map $\Psi$ immediately.
\end{proof}

 \begin{lemma} \label{SachenChen} (\cite{GRuzsa}, Theorem 5.1.8) \  Let $G$ be a cyclic group of order $n\geq 3$. If $T\in \mathcal{F}(G)$ is zero-sum free of length at least $\lfloor\frac{n}{2}\rfloor+1$, then $T$ is g-smooth for some $g\in G$ with ${\rm ord}(g)=n$.
 \end{lemma}

Now we are in a position to give the main theorem.

\begin{theorem}\label{Theorem behaving sequence in semigroup} \  For integers $k,n\geq 1$, let $T\in \mathcal{F}({\rm C}_{k;  n})$ be a sequence of length
\begin{equation}\label{equation length for T}
|T|\geq \left\{ \begin{array}{ll}
\left\lfloor\frac{(\left\lceil\frac{k}{n}\right\rceil+1) n}{2}\right\rfloor, & \textrm{if $k>n$;}\\
\\
\left\lfloor\frac{n}{2}\right\rfloor+1, & \textrm{otherwise.}\\
\end{array} \right.
\end{equation}
Then $T$ is idempotent-sum free if and only if one of the following two conditions holds:

(i) $\mathop{\bullet}\limits_{a\mid T} {\rm ind}(a)\in \mathcal{F}(\mathbb{Z})$ is $1$-smooth with $\sum\limits_{a\mid T} {\rm ind}(a)\leq \left\lceil\frac{k}{n}\right\rceil n-1$ in the case of $k>n$;

(ii) $\mathop{\bullet}\limits_{a\mid T} ({\rm ind}(a)+n\mathbb{Z}) \in \mathcal{F}(\mathbb{Z}\diagup n \mathbb{Z})$ is $g$-smooth for some $g\in  \mathbb{Z}\diagup n \mathbb{Z}$ with ${\rm ord}(g)=n$ in the case of $k\leq n$.
\end{theorem}

\begin{proof} The sufficiency of the theorem follows from Definition \ref{Definition behaving sequence} and Lemma \ref{Lemma product condition containing idmepotent}.

For $k>n$, the necessity follows from Lemma \ref{prop in cyclic semigroup} because the sequences meeting any one of Conditions (ii)-(v) have length exactly $\frac{(\left\lceil\frac{k}{n}\right\rceil+1) n}{2}-1<\left\lfloor\frac{(\left\lceil\frac{k}{n}\right\rceil+1) n}{2}\right\rfloor$.

Suppose $k\leq n$.
By Lemma \ref{lemma case of k leq n}, then $\Psi(T)=\mathop{\bullet}\limits_{a\mid T} ({\rm ind}(a)+n\mathbb{Z})\in \mathcal{F}(\mathbb{Z}\diagup n \mathbb{Z})$ is a zero-sum free sequence with length $|\Psi(T)|=|T|\geq \left\lfloor\frac{n}{2}\right\rfloor+1$. Note that $n\geq 3$, since otherwise $n\in \{1,2\}$ then $\left\lfloor\frac{n}{2}\right\rfloor+1=n={\rm D}(\mathbb{Z}\diagup n \mathbb{Z})$ which is a contradiction with $\Psi(T)$ being zero-sum free. Then the necessity follows from Lemma \ref{SachenChen} immediately.
 \end{proof}

\begin{remark} \ We remark that the values $t$ in \eqref{equation length for T} of Theorem \ref{Theorem behaving sequence in semigroup} are best possible in general to ensure an idempotent-sum free sequence $T$ over ${\rm C}_{k;  n}$ of length $|T|\geq t$ has the desired smooth sequence structure. The reason is as follows.

For $k>n$,  we see that the sequence $T$ meeting any one of Conditions (ii)-(v) in Lemma \ref{prop in cyclic semigroup} has length exactly $\left\lfloor\frac{(\left\lceil\frac{k}{n}\right\rceil+1) n}{2}\right\rfloor-1$ and satisfies $\mathop{\bullet}\limits_{a\mid T}{\rm ind}(a)$ is not $1$-smooth.
For $k\leq n$, one can check that the following idempotent-sum free sequence $V$ of length exactly $\left\lfloor\frac{n}{2}\right\rfloor$ and the sequence $\mathop{\bullet}\limits_{a\mid V} ({\rm ind}(a)+n\mathbb{Z})\in \mathcal{F}(\mathbb{Z}\diagup n \mathbb{Z})$ is not smooth:
$$\mathop{\bullet}\limits_{a\mid V}{\rm ind}(a)=\left\{ \begin{array}{ll}
1^{[\frac{n-5}{2}]}\cdot (\frac{n+3}{2})^{[2]}, & \textrm{if $n\geq  8$ and $n\equiv 1 \pmod 2$;}\\
1^{[\frac{n-4}{2}]}\cdot (\frac{n+2}{2})^{[2]}, & \textrm{if $n\geq  8$ and $n\equiv 0 \pmod 2$.}\\
\end{array} \right.$$
\end{remark}

\section{Concluding section}

Theorem \ref{Theorem behaving sequence in semigroup} asserts that
if an idempotent-sum free sequence $T$ over a cyclic semigroup ${\rm C}_{k;  n}$ has the length over `approximately' a half of the size of the cyclic semigroup, then $T$ will have a smooth sequence structure. Although the quantities in \eqref{equation length for T} are best possible {\sl in general}, it can be better for specific $k$ and $n$. So, one natural Ramsey-type question arises:
For particular $k$ and $n$,  what is the smallest positive integer $\ell$ such that every idempotent-sum free sequence $T$ over ${\rm C}_{k;  n}$ of length $|T|\geq \ell$ will yield a smooth sequence structure given as Theorem \ref{Theorem behaving sequence in semigroup}?
This type of question has been investigated by S.T. Chapman, M. Freeze and W.W. Smith \cite{ChFrSmi,ChSmi}, W. Gao \cite{GaoInteger}, and P. Yuan \cite{Yuan} for minimal zero-sum sequences over
finite cyclic groups, which is formulated as the  invariant $\mathsf{I}(\cdot)$ below.

\noindent {\bf Definition A.} \ (\cite{GRuzsa}, Definition 5.1.1 and Lemma 5.1.2) {\sl
Let $G$ be a cyclic group of order $n$.

(i) For any nonzero element $g\in G$ and for any sequence $T=(n_1g)\cdot\ldots\cdot (n_{\ell} g)$, where $\ell\in \mathbb{N}\cup \{0\}$ and $n_1,\ldots,n_{\ell}\in [1, {\rm ord}(g)]$, we define $\|T\|_g=\frac{n_1+\cdots+n_{\ell}}{{\rm ord}(g)};$

(ii) For any $T\in \mathcal{F}(G)$, we call ${\rm ind}(T)=\min \{\|T\|_g: g\in G \mbox{ with } {\rm ord}(g)=n\}\in \mathbb{Q}_{\geq 0}$ the {\rm index} of $T$;

(iii) Define $\mathsf{I}(G)$ to be the smallest integer $\ell\in \mathbb{N}$ such that every minimal zero-sum sequence $T\in \mathcal{F}(G)$ of length $|T|\geq \ell$ satisfies ${\rm ind}(T)=1$.}

The invariant $\mathsf{I}(\cdot)$ was completely determined by P. Yuan in a final critical step.

\noindent {\bf Theorem B.} (see \cite{Yuan}, or [\cite{GRuzsa}, Corollary 5.1.9]) \  Let $G$ be a cyclic group of order $n \geq 1$. If $n\in\{1,2,3,4,5,7\}$ then
$\mathsf{I}(G)=1$, and otherwise we have $\mathsf{I}(G)=\lfloor\frac{n}{2}\rfloor+2$.

Now we
formulate two invariants on the Ramsey-type question associated with idempotent-sum free and minimal idempotent-sum sequences over finite cyclic semigroups.

\begin{definition}\label{definition minimal idempotent-sum} \ For $\max(k,n)>1$ we define ${\rm Smo}({\rm C}_{k;  n})$  \ [resp. ${\rm \widehat{Smo}}({\rm C}_{k;  n})$] \ to be the least positive integer $\ell$ such that for
any minimal idempotent-sum \ [resp. idempotent-sum free] \  sequence $T\in \mathcal{F}({\rm C}_{k;  n})$ of length at least $\ell$ satisfies:

(i) If $k>n$ then the sequence $\mathop{\bullet}\limits_{a\mid T} {\rm ind}(a)\in \mathcal{F}(\mathbb{Z})$ is $1$-smooth;

(ii) If $k\leq n$ then the sequence $\mathop{\bullet}\limits_{a\mid T} ({\rm ind}(a)+n\mathbb{Z}) \in \mathcal{F}(\mathbb{Z}\diagup n \mathbb{Z})$ is zero-sum $g$-smooth [resp. $g$-smooth] for some $g\in  \mathbb{Z}\diagup n \mathbb{Z}$ with ${\rm ord}(g)=n$.
\end{definition}

We let ${\rm Smo}({\rm C}_{1;  1})=1$ and ${\rm \widehat{Smo}}({\rm C}_{1;  1})=0$. The following Lemma will illustrate us why the invariant ${\rm Smo}({\rm C}_{k;  n})$ coincides with $\mathsf{I}(\mathbb{Z}\diagup n \mathbb{Z})$ for the case of $k\leq n$ with $n=6$ or $n\geq 8$.

\begin{lemma}\label{Lemma relation index and smooth} \ For $1\leq k\leq n$, let $T\in \mathcal{F}({\rm C}_{k;  n})$ be a minimal idempotent-sum sequence. Then,

(i) If $\mathop{\bullet}\limits_{a\mid T} ({\rm ind}(a)+n\mathbb{Z})\in \mathcal{F}(\mathbb{Z}\diagup n \mathbb{Z})$ is zero-sum smooth then ${\rm ind}(\mathop{\bullet}\limits_{a\mid T} ({\rm ind}(a)+n\mathbb{Z}))=1$;

(ii) If $|T|>\frac{n}{2}$ and ${\rm ind}(\mathop{\bullet}\limits_{a\mid T} ({\rm ind}(a)+n\mathbb{Z}))=1$ then  $\mathop{\bullet}\limits_{a\mid T} ({\rm ind}(a)+n\mathbb{Z})$ is zero-sum smooth;

(iii) ${\rm Smo}({\rm C}_{k;  n})\geq \mathsf{I}(\mathbb{Z}\diagup n \mathbb{Z})$, moreover, if $n=6$ or $n\geq 8$ then ${\rm Smo}({\rm C}_{k;  n})=\mathsf{I}(\mathbb{Z}\diagup n \mathbb{Z})=\lfloor\frac{n}{2}\rfloor+2$.
\end{lemma}

\begin{proof} \  (i). The conclusion follows from Definition \ref{Definition behaving sequence} and Definition A.

(ii).  Since ${\rm ind}(\mathop{\bullet}\limits_{a\mid T} ({\rm ind}(a)+n\mathbb{Z}))=1$, it follows from Definition A that  $\|\mathop{\bullet}\limits_{a\mid T} ({\rm ind}(a)+n\mathbb{Z})\|_g=1$ for some $g\in \mathbb{Z}\diagup n \mathbb{Z}$
with ${\rm ord}(g)=n$, i.e.,  $\mathop{\bullet}\limits_{a\mid T} ({\rm ind}(a)+n\mathbb{Z})
=(n_1g)\cdot\ldots\cdot (n_{t} g)$ where $t=|T|$, $n_1,\ldots,n_{t}\in [1, n]$ and $\sum\limits_{i=1}^{t} n_i=n$. Since $t>\frac{n}{2}$, it follows from Lemma \ref{Lemma behaving for integers} that $(\mathop{\bullet}\limits_{i\in [1,t]} n_i)\in \mathcal{F}(\mathbb{Z})$ is $1$-smooth
and $\sum(\mathop{\bullet}\limits_{i\in [1,t]} n_i)=[1,n]$,
which implies $\sum(\mathop{\bullet}\limits_{a\mid T} ({\rm ind}(a)+n\mathbb{Z}))=\sum((n_1g)\cdot\ldots\cdot (n_{t} g))=\{g,2g,\ldots,ng\}$.
Hence, the sequence $\mathop{\bullet}\limits_{a\mid T} ({\rm ind}(a)+n\mathbb{Z})\in \mathcal{F}(\mathbb{Z}\diagup n \mathbb{Z})$ is zero-sum $g$-smooth.

(iii). The conclusion ${\rm Smo}({\rm C}_{k;  n})\geq \mathsf{I}(\mathbb{Z}\diagup n \mathbb{Z})$ follows from (i). Say $n=6$ or $n\geq 8$. By Theorem B, $\mathsf{I}(\mathbb{Z}\diagup n \mathbb{Z})=\lfloor\frac{n}{2}\rfloor+2>\frac{n}{2}$. Then ${\rm Smo}({\rm C}_{k;  n})=\mathsf{I}(\mathbb{Z}\diagup n \mathbb{Z})$ follows from (ii).
\end{proof}

Together with the following observation, we shall have all ingredients to find the values of ${\rm \widehat{Smo}}({\rm C}_{k;  n})$ and ${\rm Smo}({\rm C}_{k;  n})$ for cyclic semigroups ${\rm C}_{k;  n}$.

\begin{lemma} \label{Lemma last lemma}\  Let $H$ be a sequence of positive integers of length at least $2$, and let $h$ be one minimal term of $H$. If $H\cdot h^{[-1]}$ is $1$-smooth, so is $H$.
\end{lemma}

\begin{theorem}\label{propositive k leq n} \ Let $k,n$ be positive integers. Then the following conclusions hold:

(i) \ For $k\leq n$, if $n=5$ then ${\rm Smo}({\rm C}_{k;  n})=3$ and ${\rm \widehat{Smo}}({\rm C}_{k;  n})=1$, and otherwise we have
$${\rm Smo}({\rm C}_{k;  n})-1={\rm \widehat{Smo}}({\rm C}_{k;  n})=\left\{ \begin{array}{ll}
\lfloor\frac{n}{2}\rfloor, & \textrm{if $n\leq 4$ or $n=7$;}\\
\lfloor\frac{n}{2}\rfloor+1, & \textrm{if $n=6$ or $n\geq 8$.}\\
\end{array} \right.$$

(ii) \ For $k>n$, then
${\rm Smo}({\rm C}_{k;  n})\leq {\rm \widehat{Smo}}({\rm C}_{k;  n})+1$, moreover,
$$\left\{ \begin{array}{ll}
\frac{\left\lceil\frac{k}{n}\right\rceil n}{2}+1\leq {\rm \widehat{Smo}}({\rm C}_{k;  n})\leq \left\lceil\frac{(\left\lceil\frac{k}{n}\right\rceil+1) n}{2}\right\rceil-1, & \textrm{if $n\geq 3$ and $\left\lceil\frac{k}{n}\right\rceil n\equiv 0\pmod 2$;}\\
{\rm \widehat{Smo}}({\rm C}_{k;  n})=\left\lfloor\frac{(\left\lceil\frac{k}{n}\right\rceil+1) n}{2}\right\rfloor, & \textrm{otherwise,}\\
\end{array} \right.$$
and
$$\left\{ \begin{array}{ll}
\frac{\left\lceil\frac{k}{n}\right\rceil n}{2}+1\leq {\rm Smo}({\rm C}_{k;  n})\leq \left\lceil\frac{(\left\lceil\frac{k}{n}\right\rceil+1) n}{2}\right\rceil, & \textrm{if $n\geq 3$ and $\left\lceil\frac{k}{n}\right\rceil n\equiv 0\pmod 2$;}\\
\\
{\rm Smo}({\rm C}_{k;  n})=\left\lfloor\frac{(\left\lceil\frac{k}{n}\right\rceil+1) n}{2}\right\rfloor, & \textrm{if $n=2$;}\\
\\
{\rm Smo}({\rm C}_{k;  n})=\left\lfloor\frac{(\left\lceil\frac{k}{n}\right\rceil+1) n}{2}\right\rfloor+1, & \textrm{otherwise.}\\
\end{array} \right.$$
\end{theorem}

\begin{proof} (i).  By the definition, ${\rm \widehat{Smo}}({\rm C}_{1;  1})=0$ and ${\rm Smo}({\rm C}_{1;  1})=1$. Say $n\geq 2$. To calculate ${\rm Smo}({\rm C}_{k;  n})$, by Lemma \ref{Lemma relation index and smooth} (iii), it remains to consider the case of $n\in \{2,3,4,5,7\}$. Take a sequence $W\in \mathcal{F}({\rm C}_{k;  n})$, where
$$\mathop{\bullet}\limits_{a\mid W}{\rm ind}(a)=\left\{ \begin{array}{ll}
n, & \textrm{if $n\in \{2,3\}$;}\\
1\cdot (n-1), & \textrm{if $n\in \{4,5\}$;}\\
1\cdot 1\cdot 5, & \textrm{if $n=7$.}\\
\end{array} \right.$$ It is routine to check that $W$ is a minimal idempotent-sum sequence and the sequence $\mathop{\bullet}\limits_{a\mid W} ({\rm ind}(a)+n\mathbb{Z})\in \mathcal{F}(\mathbb{Z}\diagup n \mathbb{Z})$ is not zero-sum $g$-smooth for any $g\in  \mathbb{Z}\diagup n \mathbb{Z}$ with ${\rm ord}(g)=n$, which implies
\begin{equation}\label{equation smo geq|W|+1}
{\rm Smo}({\rm C}_{k;  n})\geq |W|+1=
\lfloor\frac{n}{2}\rfloor+1  \ \mbox{ for } \ n\in \{2,3,4,5,7\}.
\end{equation}
On the other hand, let $T\in \mathcal{F}({\rm C}_{k;  n})$ be a minimal idempotent-sum sequence such that
\begin{equation}\label{equation |T|>n/2}
|T|\geq \lfloor\frac{n}{2}\rfloor+1.
\end{equation}
 By Lemma \ref{lemma case of k leq n} and Theorem B, see that $\mathop{\bullet}\limits_{a\mid T} ({\rm ind}(a)+n\mathbb{Z})\in \mathcal{F}(\mathbb{Z}\diagup n \mathbb{Z})$ is a minimal zero-sum sequence with ${\rm ind}(\mathop{\bullet}\limits_{a\mid T} ({\rm ind}(a)+n\mathbb{Z}))=1$. By \eqref{equation |T|>n/2} and Lemma \ref{Lemma relation index and smooth} (ii), we derive that $\mathop{\bullet}\limits_{a\mid T} ({\rm ind}(a)+n\mathbb{Z})$ is zero-sum smooth and so ${\rm Smo}({\rm C}_{k;  n})\leq \lfloor\frac{n}{2}\rfloor+1$, combined with \eqref{equation smo geq|W|+1}, we have that ${\rm Smo}({\rm C}_{k;  n})=
\lfloor\frac{n}{2}\rfloor+1  \ \mbox{ where } \ n\in \{2,3,4,5,7\}$.

Next we figure out the value of ${\rm \widehat{Smo}}({\rm C}_{k;  n})$.
If $n\leq 5$ or $n=7$, the conclusion follows by exhaustive but trivial calculations. If $n=6$ or $n\geq 8$, we see that the following idempotent-sum free sequence $V\in \mathcal{F}({\rm C}_{k;  n})$ has length exactly $\lfloor\frac{n}{2}\rfloor$ and  $\mathop{\bullet}\limits_{a\mid V} ({\rm ind}(a)+n\mathbb{Z})\in \mathcal{F}(\mathbb{Z}\diagup n \mathbb{Z})$ is not $g$-smooth for any $g\in  \mathbb{Z}\diagup n \mathbb{Z}$ with ${\rm ord}(g)=n$, which implies that ${\rm \widehat{Smo}}({\rm C}_{k;  n})\geq |V|+1=\lfloor\frac{n}{2}\rfloor+1$, where $$\mathop{\bullet}\limits_{a\mid V}{\rm ind}(a)=\left\{ \begin{array}{ll}
1^{[\frac{n-5}{2}]}\cdot (\frac{n+3}{2})^{[2]}, & \textrm{if $n\geq  8$ and $n\equiv 1 \pmod 2$;}\\
1^{[\frac{n-4}{2}]}\cdot (\frac{n+2}{2})^{[2]}, & \textrm{otherwise.}\\
\end{array} \right.$$
On the other hand,  Theorem \ref{Theorem behaving sequence in semigroup} tells us that
${\rm \widehat{Smo}}({\rm C}_{k;  n})\leq \left\lfloor\frac{n}{2}\right\rfloor+1,$ and so ${\rm \widehat{Smo}}({\rm C}_{k;  n})=\left\lfloor\frac{n}{2}\right\rfloor+1$ when $n=6$ or $n\geq 8$,
completing the calculations of ${\rm \widehat{Smo}}({\rm C}_{k;  n})$.

\noindent (ii). We first calculate ${\rm \widehat{Smo}}({\rm C}_{k;  n})$. Take a sequence $V\in \mathcal{F}({\rm C}_{k;  n})$
such that
$$\mathop{\bullet}\limits_{a\mid V}{\rm ind}(a)=\left\{ \begin{array}{ll}
3\cdot 2^{[\frac{\left\lceil\frac{k}{n}\right\rceil n}{2}-1]}, & \textrm{if $n\geq 2$ and $\left\lceil\frac{k}{n}\right\rceil n\equiv 0\pmod 2$;}\\
2^{[\left\lfloor\frac{(\left\lceil\frac{k}{n}\right\rceil+1) n}{2}\right\rfloor-1]}, & \textrm{otherwise.}\\
\end{array} \right.$$
By Lemma \ref{Lemma product condition containing idmepotent}, we can check that $V$ is an idempotent-sum free sequence. Since $\mathop{\bullet}\limits_{a\mid V}{\rm ind}(a)$ is not $1$-smooth,
noting $\frac{\left\lceil\frac{k}{n}\right\rceil n}{2}+1=\left\lfloor\frac{(\left\lceil\frac{k}{n}\right\rceil+1) n}{2}\right\rfloor$ for $n=2$,
we conclude that
\begin{equation}\label{equation instant equation}
{\rm \widehat{Smo}}({\rm C}_{k;  n})\geq |V|+1=\left\{ \begin{array}{ll}
\frac{\left\lceil\frac{k}{n}\right\rceil n}{2}+1, & \textrm{if $n\geq 3$ and $\left\lceil\frac{k}{n}\right\rceil n\equiv 0\pmod 2$;}\\
\left\lfloor\frac{(\left\lceil\frac{k}{n}\right\rceil+1) n}{2}\right\rfloor, & \textrm{otherwise.}\\
\end{array} \right.
\end{equation}
By
Theorem \ref{Theorem behaving sequence in semigroup}, we derive that ${\rm \widehat{Smo}}({\rm C}_{k;  n})\leq \left\lfloor\frac{(\left\lceil\frac{k}{n}\right\rceil+1) n}{2}\right\rfloor$. Furthermore, if $n\geq 3$ and $\left\lceil\frac{k}{n}\right\rceil n\equiv 0\pmod 2$, we conclude from Lemma \ref{prop in cyclic semigroup} that every idempotent-sum free sequence $U\in \mathcal{F}({\rm C}_{k;  n})$
of length at least $\frac{(\left\lceil\frac{k}{n}\right\rceil+1) n}{2}-1$ yields that $\mathop{\bullet}\limits_{a\mid U} {\rm ind}(a)$ is $1$-smooth, i.e., ${\rm \widehat{Smo}}({\rm C}_{k;  n})\leq \left\lceil\frac{(\left\lceil\frac{k}{n}\right\rceil+1) n}{2}-1\right\rceil=\left\lceil\frac{(\left\lceil\frac{k}{n}\right\rceil+1) n}{2}\right\rceil-1$.
Combined with \eqref{equation instant equation}, we complete the calculations of  ${\rm \widehat{Smo}}({\rm C}_{k;  n})$.

To establish ${\rm Smo}({\rm C}_{k;  n})\leq {\rm \widehat{Smo}}({\rm C}_{k;  n})+1$,
let $T\in \mathcal{F}({\rm C}_{k;  n})$ be a minimal idempotent-sum sequence of length at least ${\rm \widehat{Smo}}({\rm C}_{k;  n})+1$. It suffices to show that $\mathop{\bullet}\limits_{a\mid T} {\rm ind}(a)$ is $1$-smooth. Take a term $b$ of $T$ with ${\rm ind}(b)$ minimal.
Since $T\cdot b^{[-1]}$ is an idempotent-sum free sequence of length $|T\cdot b^{[-1]}|\geq {\rm \widehat{Smo}}({\rm C}_{k;  n})$, it follows that $\mathop{\bullet}\limits_{a\mid T\cdot b^{[-1]}} {\rm ind}(a)$ is $1$-smooth. Then the conclusion follows from Lemma \ref{Lemma last lemma}  immediately.

Now we calculate ${\rm Smo}({\rm C}_{k;  n})$. Take a sequence $W\in \mathcal{F}({\rm C}_{k;  n})$ with $\mathop{\bullet}\limits_{a\mid W}{\rm ind}(a)=2^{[\ell]}$ and
$$\ell=\left\{ \begin{array}{ll}
\frac{\left\lceil\frac{k}{n}\right\rceil n}{2}, & \textrm{if $\left\lceil\frac{k}{n}\right\rceil n\equiv 0\pmod 2$;}\\
\frac{(\left\lceil\frac{k}{n}\right\rceil+1) n}{2}, & \textrm{otherwise.}\\
\end{array} \right.$$
By Lemma \ref{Lemma product condition containing idmepotent}, we see that $W$ is a minimal idempotent-sum sequence. Since $\mathop{\bullet}\limits_{a\mid W}{\rm ind}(a)$ is not $1$-smooth, we have
$${\rm Smo}({\rm C}_{k;  n})\geq |W|+1=\left\{ \begin{array}{ll}
\frac{\left\lceil\frac{k}{n}\right\rceil n}{2}+1, & \textrm{if $\left\lceil\frac{k}{n}\right\rceil n\equiv 0\pmod 2$;}\\
\frac{(\left\lceil\frac{k}{n}\right\rceil+1) n}{2}+1, & \textrm{otherwise.}\\
\end{array} \right.$$
Noting that if $n=1$ and $k=\left\lceil\frac{k}{n}\right\rceil n\equiv 0\pmod 2$ then $\frac{\left\lceil\frac{k}{n}\right\rceil n}{2}+1=\left\lfloor\frac{(\left\lceil\frac{k}{n}\right\rceil+1) n}{2}\right\rfloor+1$ and that if $n=2$ then $\frac{\left\lceil\frac{k}{n}\right\rceil n}{2}+1=\left\lfloor\frac{(\left\lceil\frac{k}{n}\right\rceil+1) n}{2}\right\rfloor$, combined with the obtained inequality ${\rm Smo}({\rm C}_{k;  n})\leq {\rm \widehat{Smo}}({\rm C}_{k;  n})+1$ and the result for ${\rm \widehat{Smo}}({\rm C}_{k;  n})$,
we conclude that $$\left\{ \begin{array}{ll}
\frac{\left\lceil\frac{k}{n}\right\rceil n}{2}+1\leq {\rm Smo}({\rm C}_{k;  n})\leq \left\lceil\frac{(\left\lceil\frac{k}{n}\right\rceil+1) n}{2}\right\rceil, & \textrm{if $n\geq 3$ and $\left\lceil\frac{k}{n}\right\rceil n\equiv 0\pmod 2$;}\\
\left\lfloor\frac{(\left\lceil\frac{k}{n}\right\rceil+1) n}{2}\right\rfloor\leq {\rm Smo}({\rm C}_{k;  n})\leq \left\lfloor\frac{(\left\lceil\frac{k}{n}\right\rceil+1) n}{2}\right\rfloor+1, & \textrm{if $n=2$;}\\
{\rm Smo}({\rm C}_{k;  n})=\left\lfloor\frac{(\left\lceil\frac{k}{n}\right\rceil+1) n}{2}\right\rfloor+1, & \textrm{otherwise.}\\
\end{array} \right.$$
To complete the proof, it remains to show that  ${\rm Smo}({\rm C}_{k;  n})\leq \left\lfloor\frac{(\left\lceil\frac{k}{n}\right\rceil+1) n}{2}\right\rfloor$ when $n=2$.
Assume to the contrary that ${\rm Smo}({\rm C}_{k;  n})> \left\lfloor\frac{(\left\lceil\frac{k}{n}\right\rceil+1) n}{2}\right\rfloor$ for $n=2$. Take a minimal idempotent-sum sequence $L\in \mathcal{F}({\rm C}_{k;  n})$ of length $|L|={\rm Smo}({\rm C}_{k;  n})-1\geq \left\lfloor\frac{(\left\lceil\frac{k}{n}\right\rceil+1) n}{2}\right\rfloor$
such that $\mathop{\bullet}\limits_{a\mid L}{\rm ind}(a)$ is not $1$-smooth.
Take a term $b$ of $L$ with ${\rm ind}(b)$ minimal. By Lemma \ref{Lemma last lemma}, the sequence $\mathop{\bullet}\limits_{a\mid L\cdot b^{[-1]}} {\rm ind}(a)$ is not $1$-smooth.
Since $L\cdot b^{[-1]}$ is an idempotent-sum free sequence of length $|L\cdot b^{[-1]}|\geq \left\lfloor\frac{(\left\lceil\frac{k}{n}\right\rceil+1) n}{2}\right\rfloor-1=\frac{(\left\lceil\frac{k}{n}\right\rceil+1) n}{2}-1$, it follows from
 Lemma \ref{prop in cyclic semigroup} that
$\mathop{\bullet}\limits_{a\mid L\cdot b^{[-1]}} {\rm ind}(a)=z\cdot 2^{\left[\lceil\frac{k}{2}\rceil-1\right]}$ with $z\geq 3$ and $z\equiv 1\pmod 2$. Since $\lceil\frac{k}{2}\rceil>1$, we can take a term $c$ of $L$ with ${\rm ind}(c)=2$. Combined with Lemma \ref{Lemma product condition containing idmepotent}, we verify that
$\sum\limits_{a\mid L\cdot c^{[-1]}}{\rm ind}(a)=2(\lceil\frac{k}{2}\rceil-2)+z+{\rm ind}(b)\geq 2(\lceil\frac{k}{2}\rceil-2)+3+1\geq k$ and $\sum\limits_{a\mid L\cdot c^{[-1]}}{\rm ind}(a)= \sum\limits_{a\mid L}{\rm ind}(a)-2\equiv\sum\limits_{a\mid L}{\rm ind}(a)\equiv 0\pmod 2$, and so $L\cdot c^{[-1]}$ is a nonempty {\sl proper} idempotent-sum subsequence of $L$, which contradicts with $L$ being a minimal idempotent-sum sequence, completing the proof.
\end{proof}

We close this paper with the following problem.

\begin{problem} \ Determine ${\rm \widehat{Smo}}({\rm C}_{k;  n})$ and ${\rm Smo}({\rm C}_{k;  n})$ when $k>n\geq 3$ and $\left\lceil\frac{k}{n}\right\rceil n\equiv 0\pmod 2$.
\end{problem}

\noindent {\bf Acknowledgements}

\noindent
The research is supported by NSFC (grant no. 11971347, 11501561).

\end{document}